\documentclass[12pt, reqno]{amsart}

\usepackage{amsfonts}
\usepackage{amssymb}
\usepackage{amsthm}
\usepackage{amsmath}
\usepackage{a4wide}
\usepackage{mathrsfs}
\usepackage{epsfig}
\usepackage{palatino}
\usepackage{tikz,fp,ifthen}
\usepackage{float}
\usepackage{graphicx,cite}
\usepackage{comment}
\usepackage{accents}

\usepackage{graphicx}

\usepackage{color}
\definecolor{citegreen}{rgb}{0,0.6,0}
\definecolor{refred}{rgb}{0.8,0,0}
\usepackage[colorlinks, citecolor=citegreen,linkcolor=refred]{hyperref}

\usepackage{mathtools}
\mathtoolsset{showonlyrefs}

\newtheorem{thm}{Theorem}[section]
\newtheorem{lem}[thm]{Lemma}
\newtheorem{prop}[thm]{Proposition}

\theoremstyle{definition}
\newtheorem{defn}[thm]{Definition}

\theoremstyle{remark}
\newtheorem{rem}[thm]{Remark}

\numberwithin{equation}{section}

\def\SSS{{{\mathcal S}}}

\def\F{\mathcal F}
\def\LL{\mathcal L}
\def\UU{\mathcal U}
\def\E{\mathbb E}
\def\R{\mathbb R}

\def\R{{{\mathbb R}}}

\def\NN{\mathbb N}

\def\TTT{\mathcal T}
\newcommand{\oX}{\overline{X}}
\newcommand{\oY}{\overline{Y}}
\newcommand{\omu}{\overline{\mu}}
\newcommand{\W}{\mathcal{W}}
\newcommand{\Norm}[2]{\left\Vert #1 \right\Vert_{#2}}
\renewcommand{\div}{\mathrm{div}}

\usetikzlibrary{backgrounds}
\usetikzlibrary{decorations.pathmorphing,backgrounds,fit,calc,through}
\usetikzlibrary{arrows}
\usetikzlibrary{shapes,decorations,shadows}
\usetikzlibrary{fadings}
\usetikzlibrary{patterns}
\usetikzlibrary{mindmap}
\usetikzlibrary{decorations.text}
\usetikzlibrary{decorations.shapes}

\begin{document}

\title{Mean field sparse optimal control of systems with additive white noise}

\author{Giacomo Ascione}
\address{Giacomo Ascione\\
Scuola Superiore Meridionale,
Universit\`a di Napoli, Largo San Marcellino 10, 80138 Napoli,
Italy}
\email{giacomo.ascione@unina.it}

\author{Daniele Castorina}
\address{Daniele Castorina\\
Dipartimento di Matematica e Applicazioni,
Universit\`a di Napoli, Via Cintia, Monte S. Angelo 80126 Napoli,
Italy}
\email{daniele.castorina@unina.it}

\author{Francesco Solombrino}
\address{Francesco Solombrino\\
Dipartimento di Matematica e Applicazioni,
Universit\`a di Napoli, Via Cintia, Monte S. Angelo 80126 Napoli,
Italy}
\email{francesco.solombrino@unina.it}

\begin{abstract}  
We analyze the problem of controlling a multi-agent system with additive white noise through parsimonious interventions on a selected subset of the agents (leaders). For such a controlled system with a SDE constraint, we introduce a rigorous limit process towards an infinite dimensional optimal control problem constrained by the coupling of a system of ODE for the leaders with a McKean-Vlasov-type SDE, governing the dynamics of the prototypical follower. The latter is, under some assumptions on the distribution of the initial data, equivalent with a (nonlinear parabolic) PDE-ODE system. The derivation of the limit mean-field optimal control problem is achieved by linking the mean-field limit of the governing equations together with the $\Gamma$-limit of the cost functionals for the finite dimensional problems.
\end{abstract}

\maketitle
\tableofcontents

\section{Introduction}\label{intro}
In recent years, multi-agent systems have been widely used to describe several phenomena, such as, for instance, flocking and cell aggregation in biology \cite{hofbauer1998evolutionary,camazine2020self,keller1970initiation,cucker2007emergent}, chemical networks \cite{lim2020quantitative,mozgunov2018review,oelschlager1989derivation}, human interaction in social sciences \cite{during2009boltzmann,toscani2006kinetic}, neuronal spike dynamics \cite{delarue2015particle,flandoli2019mean}, cooperative robots \cite{chuang2007multi,leonard2001virtual,perea2009extension,auletta2022herding} and so on. Such systems are analytically and computationally manageable as soon as the number of involved agents is not too large. However, such a number is usually very large in real life applications.\\
In order to overcome this issue, an effective approach consists in the approximation procedure obtained by sending the number of agents $N$ to $\infty$, in place of considering just a large number of them. This procedure is known in the literature as \textit{mean-field limit} \cite{kac1956foundations,mckean1967propagation,carmona2018probabilistic}. Clearly, the mean-field limit of a system of particles described by interacting Ordinary Differential Equations (ODEs) leads to a Partial Differential Equation (PDE), usually hyperbolic and nonlinear, for the density of particles in a certain region of the space. From this point of view, we are simplifying the problem by reducing a very large system of coupled ODEs to a single PDE. On the other hand, the clear improvement in \textit{manageability} is balanced by the inevitable loss in \textit{accuracy}, since the obtained model is just an approximation. Furthermore, the position and/or the velocity of the particles in their full complexity could easily render impossible their exact study. In this regard, it is useful to include a white noise in the system, reducing its complexity.

Mean-field limits for systems of Stochastic Differential Equations (SDEs) have been widely studied in literature (see, for instance, \cite{sznitman1991topics,meleard1996asymptotic,jabin2017mean} and references therein) and present some different characteristics with respect to the ones that appear in the ODE case. Precisely, the limit PDE is no more hyperbolic, but it gains a diffusive term, thus becoming parabolic. Moreover, its solution represents the probability of finding a generic particle in a region of the space. Analogously to what happens in the ODE case, the dynamics of the \textit{limit particle} is described by a (stochastic) McKean-Vlasov equation and the limit PDE is the associated (nonlinear) Fokker-Planck equation.

In the applications, we are not satisfied by the knowledge of the dynamics, but we are also interested in controlling it. Clearly this can be achieved by introducing some control functions in the ODE system. Such an extension is \textit{practically unmanageable}, also from a purely numerical point of view, if the number of the agents is too large. Again, the mean-field limit approach turns out to be a successful tool to study suitable approximations of such a control. In the setting of the optimal control theory, this is done by combining the mean-field limit with the $\Gamma$-limit tool, as done in \cite{fornasier2014mean}.

The previous result relies on the fact that one is applying the control directly on each agent of the system. This is, however, not possible in different applications, as for instance crowd control in emergency situations. In such a case, some \textit{special agents}, called \textit{herders}, are required to drive \textit{passive agents}, the \textit{herd}, towards a designated confinement. 

The herding problem has been recently developed, see for instance \cite{albi2016invisible,albi2017mean,albi2021mean,burger2020instantaneous,lien2004shepherding,lien2005shepherding,sebastian2021multi,pierson2017controlling,auletta2022herding}. The main features of the model are the following: on the one hand, the control can be applied only on the herders' dynamics; on the other hand, the number of agents in the herd is usually much larger than the number of herders. Once again, we can describe both the herd and the herders via a system of coupled ODEs and imagine that the number $N$ of agents in the herd is sent to $\infty$, while the number $m$ of herders is fixed. At the limit, we obtain a system composed by a single hyperbolic nonlinear PDE coupled with $m$ controlled ODEs. This is done by exploiting both the mean-field limit and the $\Gamma$-limit of the control cost functionals, as done in \cite{Fornasier2014}.

Clearly, such an approach cannot be used directly if we introduce uncertainty in the dynamics of the herd. This problem has been, for instance, considered in \cite{auletta2022herding}, where it is solved numerically for a small herd. Since the dynamics of the herd, as described in \cite[Eq.(3)]{auletta2022herding}, is subject to an additive noise term, one cannot use the results contained in \cite{Fornasier2014} to approximate the system by mean-field approach.\\

Inspired by this, in the present paper we investigate the mean-field approach in the previously described herding problem, where the herd exhibits an additive white noise. Without going into details, we consider the discrete model
\begin{equation}\label{microdiscrintro}
	\begin{cases}
		\displaystyle d\, X^n (t) = \frac{1}{N} \sum_{l=1}^{N} H_1 (X^l(t) - X^n(t)) + \frac{1}{m} \sum_{j=1}^{m} K_1 ( Y^{j} (t) -  X^n (t)) \, dt+ \sqrt{2 \sigma} \, dW^n(t)\\
		\displaystyle \frac{d}{dt} Y^i (t) = \frac{1}{N} \sum_{l=1}^{N} K_2 (Y^i (t)-X^l(t)) + \frac{1}{m} \sum_{j=1}^{m} H_2 ( Y^{j} (t) - Y^i (t)) + u_N^i(X^1(t),..,X^n(t)).\\
	\end{cases}
\end{equation}
where the position of the agents in the herd is described by $X^n$, $n=1,\dots,N$, while the herders' position is described by $Y^i$, $i=1,\dots,m$ and $u_N^i$ is a suitable control function, minimizing a certain cost functional $\mathcal{F}_N$. It is clear that if $\sigma=0$ we go recover the system in \cite{Fornasier2014}. Then we study the problem as $N \to \infty$. Precisely, we show that the empirical measure $\sum_{n=1}^N \delta_{X^n(t)}$ converges in Wasserstein distance to a measure $\omu$ and $Y^i$, $i=1,\dots,m$ converge in expectation to some deterministic functions $\oY^i$, which turn out to solve the PDE/ODE system
\begin{equation}\label{pdeodeintro}
	\begin{cases}
		(\partial_t - \sigma \Delta) \ \omu_t = - \div \left( (H_1 \ast \overline{\mu}_t (\cdot) + \frac{1}{m} \sum_{j=1}^{m} K_1 (\oY^{j} (t) - \cdot))\omu_t \right)\\
		\displaystyle \frac{d}{dt} \overline{Y}^i (t) = K_2 \ast \overline{\mu}_t (\oY^i (t)) + \frac{1}{m} \sum_{j=1}^{m} H_2 (\oY^{j} (t) - \oY^i (t)) + u^i(t,\omu_t). & i=1,..,m,
	\end{cases}
\end{equation}
where $u^i$ minimizes a cost functional $\mathcal{F}$ that is the $\Gamma$-limit of $\mathcal{F}_N$. To prove the mean-field limit, as it is usual in the stochastic setting, one has to prove a \textit{propagation of chaos} result, i.e. the fact that the (initially coupled) positions of the agents in the herd become \textit{independent as $N \to \infty$}. Precisely, $X^n$, $n=1,\dots,j$, converge in expectation to $j$ independent copies of the solution $\oX$ of the system
\begin{equation}\label{microintro}
 	\begin{cases}
 		\displaystyle d\,\overline{X} (t) = H_1 \ast \overline{\mu}_t (\oX(t)) + \frac{1}{m} \sum_{j=1}^{m} K_1 (\oY^{j} (t) - \oX (t)) \, dt+ \sqrt{2 \sigma} \, dW(t)\\
 		\displaystyle \frac{d}{dt} \overline{Y}^i (t) = K_2 \ast \overline{\mu}_t (\oY^i (t)) + \frac{1}{m} \sum_{j=1}^{m} H_2 (\oY^{j} (t) - \oY^i (t)) + u^i(t,\omu_t). & i=1,..,m\\
 		\omu={\rm Law}(\oX), \ \omu_t=({\rm ev}_t)_\sharp \omu.
 	\end{cases}
\end{equation}
The paper is organized as follows. Section \ref{prelim} is devoted to discuss some preliminaries and to set up the notations. In Section \ref{existcon} we first prove the well-posedness of the system \eqref{microintro} and the fact that \eqref{pdeodeintro} is its Fokker-Planck equation. In particular, we prove the equivalence of the well-posedness of \eqref{pdeodeintro} and the one of \eqref{microintro} under suitable assumptions on the initial data. Then, we prove the propagation of chaos result, that leads directly to the mean-field limit. At this point, in Section \ref{optcon} we are able to develop the limit optimal control theory. Precisely, we first prove separately the well-posedness of the discrete optimal control system and of the limit one. Then, we prove the $\Gamma$-limit result on the cost functionals, thus showing that the optimal control in the limit setting is a good approximation of the control in the discrete one for a large herd.\\
Clearly, this is a first step in such a direction and then some restrictive hypotheses on the interaction and the controls have been imposed. In future works we plan to relax such restrictions as well as considering problems with multiplicative white noise, second order problems with white noise and controls in the velocity terms and problems with a different kind of additive noise.

\textbf{Acknowledgements}: G. Ascione has been partially supported by MIUR - PRIN 2017, project Stochastic Models for Complex Systems, No. 2017JFFHSH. D. Castorina has received funding from the research grant ``BIOMASS'' from the University of Naples Federico II - ``Finanziamento della Ricerca di Ateneo (FRA) - Linea B. The work of F. Solombrino is part of the MIUR - PRIN 2017, project Variational Methods for Stationary and Evolution Problems with Singularities and Interfaces, No. 2017BTM7SN.  He also acknowledges support by project Starplus 2020 Unina Linea 1 "New challenges in the variational modeling of continuum mechanics" from the University of Naples Federico II and Compagnia di San Paolo. Finally, all the authors are partially supported by Gruppo Nazionale per l’Analisi Matematica, la Probabilitá e le loro Applicazioni (GNAMPA-INdAM).

\section{Preliminaries and notations}\label{prelim}

In this section we will set up the main assumptions on our model, discuss some preliminaries and lay out our notations and terminology.\\
\subsection{General notation}
For any horizon $T>0$ and any complete metric space $B$, we let $C^0([0,T];B)$ be the space of $B$-valued continuous functions over $[0,T]$. Moreover, for any $t \in [0,T]$ we denote by ${\rm ev}_t$ the evaluation map, \textit{i.e.} the map ${\rm ev}_t (f) = f(t)$ for any $f \in C^0([0,T]; \R^d)$. For any $p \geq 1$, we let $\W_p (B)$ be the $p$-Wasserstein space, \textit{i.e.} the space of Borel probability measures $\mu$ on $B$, such that 
\begin{equation}
M^p_p(\mu, x_0 )=\int_{B} (d(x,x_0))^p \, d\mu < \infty, 
\end{equation} 
where $d$ is the metric on $B$ and $x_0 \in B$ is fixed (see \cite[Definition 6.4]{villani}).  Let $\mu,\nu \in \W_p (B)$; we say that a Borel probability measure $\gamma$ on $B\times B$ is a \textit{coupling} of $\mu$ and $\nu$ if for for any Borel set $A\subset B$ it holds $\gamma(A\times B) = \mu(A)$ and $\gamma(B\times A)=\nu(A)$ (see \cite[Definition 1.1]{villani}). We denote by $\Pi (\mu,\nu)$ the set of all possible couplings. Finally, it is possible to equip $\W_p(B)$ with the Wasserstein metric 
\begin{equation}
\W^p_p(\mu,\nu) = \inf_{\gamma \in \Pi (\mu,\nu)} \int_{B \times B} (d(x,y))^p \, d\gamma, 
\end{equation}
(see \cite[Definition 6.4]{villani}). Let us recall that, by \cite[Theorem 4.1]{villani}, the above infimum is actually achieved and the minimum point is called \textit{optimal coupling}. We also denote by $\sharp$ the pushforward operator. In the following, $B$ will be a Banach space, hence we can set $x_0=0$ and $d(x,y)=\|x-y\|$. We will omit $x_0$ from the above notation.\\

Fix $d \geq 1$ and let $H_i: \R^d \to \R^d$, $K_i : \R^d \to \R^d$, $1=1,2$ be globally Lipschitz continuous funtions. In what follows, without loss of generality we shall suppose that all the involved random variables are supported on a fixed filtered probability space $(\Omega, \F, \F_t, \mathbb{P})$ and we denote by $\E$ the expectation operator. Furthermore, for any metric space $B$ we denote by $\mathcal{M} (\Omega;B)$ the space of $B$-valued random variables and, for any $X \in \mathcal{M} (\Omega; B)$ we define ${\rm Law}(X) = (X)_\sharp \mathbb{P}$. When possible, we will omit the dependence on the probability space. \\
\begin{rem}\label{fuffa} Let us recall some basic properties.
	\begin{itemize}
		\item[i] Let $\nu \in \W_p (C^0([0,T];\R^d))$ and set $\nu_t=({\rm ev}_t)_\sharp \nu$. We point out the following elementary inequality
		\begin{equation*}
			M_p(\nu_t) \leq M_p (\nu)
		\end{equation*}
		for any $t \in [0,T]$. 
		\item[ii] Since $H_j,K_j,j=1,2$ are globally Lipschitz continuous, also their convolution products with $\nu_t$ are globally Lipschitz continuous with the same Lipschitz constants.
		
	\end{itemize}
\end{rem}

We also recall the following \textit{Doob's maximal inequality} in its $L^p$ form for the Brownian motion.

\begin{thm}[\cite{revuzyor}] \label{revuzunz}
	Let $W(t)$ be a $d$-dimensional Brownian motion. For any $p>1$ it holds:
	\begin{equation*}
		\E \left(\max_{0\leq t \leq T} |W(t)|^p \right) \leq \left( \frac{p}{p-1} \right)^{p} \E (|W(T)|^p)
	\end{equation*} 
\end{thm}    

We will make use of the following general result on the convergence rate of the empirical measures generated by i.i.d. random variables to their law.

\begin{thm}[\cite{fourgui}]\label{rate}
	Let $p>1$ and $\mu \in \W_p(\R^d)$. Let also $X_n$ be a sequence of i.i.d. random variables with distribution $\mu$ and $\mu_N = \frac{1}{N} \sum_{n=1}^N \delta_{X_n}$ be the corresponding empirical measures. Then there exists a constant $C=C(p,d)$ such that
	\begin{equation}
		\E\left(\W_1 (\mu_N, \mu) \right) \leq C M_p (\mu) \begin{cases} N^{-1/2} + N^{-(p-1)/p} &\text{ for } d=1,p \ne 2\\ N^{-1/2}\log(1+N) + N^{-(p-1)/p} &\text{ for } d=2,p \ne 2\\ N^{-1/d} + N^{-(p-1)/p} &\text{ for } d \geq 3,p \ne d/(d-1) \end{cases}
	\end{equation}  
\end{thm}    
\medskip
\subsection{The model system}
Throughout the paper, we will consider the following SDE-ODE system:

\begin{equation}\label{micro}
\begin{cases}
\displaystyle d\,\overline{X} (t) = H_1 \ast \overline{\mu}_t (\oX(t)) + \frac{1}{m} \sum_{j=1}^{m} K_1 (\oY^{j} (t) - \oX (t)) \, dt+ \sqrt{2 \sigma} \, dW(t)\\
\displaystyle \frac{d}{dt} \overline{Y}^i (t) = K_2 \ast \overline{\mu}_t (\oY^i (t)) + \frac{1}{m} \sum_{j=1}^{m} H_2 (\oY^{j} (t) - \oY^i (t)) + u^i(t,\omu_t). & i=1,..,m\\
\oX(0)=\oX_0, \ \oY^i (0)=\oY^i_0 & i=1,..,m\\
\omu={\rm Law}(\oX), \ \omu_t=({\rm ev}_t)_\sharp \omu
\end{cases}
\end{equation}

where $m \in \NN$ and $\sigma>0$ are fixed constants, $W$ is a $d$-dimensional Brownian motion, $\oX_0 \in L^2 (\mathbb{P})$ is $\F_0$ measurable, $\oY_0^i \in \R^d$, $i=1,..,m$, and the controls $u^i:[0,T]\times \W_1 (\R^d) \to \R^d$, $i=1,..,m$, are Carath{\'e}odory functions which are Lipschitz continuous in the second variable for any fixed $t$ and the Lipschitz contant is integrable. Whenever it is possible to choose a uniform Lipschitz constant, then we will denote it by $L$ and we will assume, without loss of generality, that such $L$ is also common to $H^i$, $K_i$, $i=1,2$.\\

\begin{defn}
We say that $(\oX,\oY^1,..,\oY^m) \in \mathcal{M}(\Omega; C^0([0,T]; \R^d) )\times (C^0 ([0,T]; \R^d))^m$  is a \textit{pathwise} (or \textit{strong}) \textit{solution} of the equation \eqref{micro} with initial data $(\oX_0,\oY_0^1,..,\oY_0^m) \in L^p (\Omega;\R^d) \times (\R^{d})^m$ and Brownian motion $W$ if, setting $\omu={\rm Law}(\oX)$ and $\omu_t=({\rm ev}_t)_\sharp \omu$, the following holds:
\begin{enumerate}
\item almost surely
\begin{equation}
\oX(t) = \oX_0 + \int_0^t \ \, \left( H_1 \ast \overline{\mu}_s (\oX(s)) + \frac{1}{m} \sum_{j=1}^{m} K_1 (\oY^{j} (s) - \oX (s)) \right) \, ds + \sqrt{2 \sigma} \, W(t), 
\end{equation}
for all $t \in [0,T]$;
\item $\oY^i$, $i=1,..,m$, is a Carath{\'e}odory solution of
\begin{equation}
\frac{d}{dt} \overline{Y}^i (t) = K_2 \ast \overline{\mu}_t (\oY^i (t)) + \frac{1}{m} \sum_{j=1}^{m} H_2 (\oY^{j} (t) - \oY^i (t)) + u^i(t,\omu_t), \quad i=1,..,m
\end{equation}
with $\oY^i (0) = \oY_0^i$, $i=1,..,m$.
\end{enumerate}

We say that \eqref{micro} admits a \textit{strong solution} if for any Brownian motion $W$ and any initial data $(\oX_0,\oY_0^1,..,\oY_0^m)$ there exists a pathwise solution $(\oX,\oY^1,..,\oY^m)$.\\
We say that the solution is \textit{pathwise} unique if, whenever we consider any two solutions $(\oX_j,\oY_j^1,..,\oY_j^m)$, $j=1,2$, given the same initial data and Brownian motion, it holds $\oX_1 = \oX_2$ almost surely as well as $(\oY_1^1,..,\oY_1^m) = (\oY_2^1,..,\oY_2^m)$.  \\
If pathwise uniqueness holds we define the solution map $S$ acting on a Brownian motion $W$ and initial data $(\oX_0,\oY_0^1,..,\oY_0^m)$ by asking that $S(W, \oX_0,\oY_0^1,..,\oY_0^m)$ is the unique pathwise solution of \eqref{micro} given the Brownian motion and initial data.
\end{defn}

\section{Existence and convergence results for the controlled system}\label{existcon}

\subsection{Well posedness for the SDE/ODE system}\label{wellpos}

We start by proving the existence of a unique pathwise solution of \eqref{micro}.

\begin{thm}\label{expath}
Equation \eqref{micro} admits strong solutions that are pathwise unique. Moreover, if $\oX_0 \in L^p(\Omega; \R^d)$, then there exists a constant $C=C(p,\oY_0^1,\dots,\oY_0^m,{\rm Law}(\oX_0),L)>0$ such that $M_p(\omu) \le C$.
\end{thm}
    
\begin{proof}
The proof is based on fixed point argument and will be divided in several steps, as in \cite{coghi}. To simplify the notation we set $Z_0 = (\oX_0,\oY_0^1,..,\oY_0^m)$.\\

\textbf{Step 1:} For any $T>0$ fix $\nu \in \W_p (C^0 ([0,T]; \R^d))$ and consider the system 

\begin{equation}\label{micro2}
\begin{cases}
\displaystyle d\,\widetilde{X} (t) = H_1 \ast \nu_t (\widetilde{X}(t)) + \frac{1}{m} \sum_{j=1}^{m} K_1 (\widetilde{Y}^{j} (t) - \widetilde{X} (t)) \, dt+ \sqrt{2 \sigma} \, dW(t)\\
\displaystyle \frac{d}{dt} \widetilde{Y}^i (t) = K_2 \ast \nu_t (\widetilde{Y}^i (t)) + \frac{1}{m} \sum_{j=1}^{m} H_2 (\widetilde{Y}^{j} (t) - \widetilde{Y}^i (t)) + u^i(t,\nu_t). & i=1,..,m\\
\widetilde{X}(0)=\oX_0, \ \widetilde{Y}^i (0)=\oY^i_0 & i=1,..,m\\
\nu_t=({\rm ev}_t)_\sharp \nu
\end{cases}
\end{equation}

We prove that \eqref{micro2} admits strong solutions that are pathwise unique and such that $M_p ({\rm Law} (\widetilde{X})) < \infty$. \\

\textbf{Proof of Step 1:} First of all observe that the second equation in \eqref{micro2} is uncoupled with the first one, and then admits global solutions by means of classical ODE theory, recalling that the involved nonlinearities are globally Lipschitz continuous.\\

Thus we now focus on the first equation in \eqref{micro2}. Let us fix $T$ and $\nu \in \W_p (C^0 ([0,T]; \R^d))$. Fix $\omega \in \Omega$ and define the map $\Phi_\omega : C^0 ([0,T]; \R^d) \to C^0 ([0,T]; \R^d)$ as, for any $\eta \in C^0 ([0,T]; \R^d)$, 
\begin{equation}
\Phi_\omega (\eta)(t) := \oX_0(\omega) + \int_0^t \left[ H_1 \ast \nu_s (\eta(s)) + \frac{1}{m} \sum_{j=1}^{m} K_1 (\widetilde{Y}^{j} (s) - \eta(s)) \right] \, ds + \sqrt{2 \sigma} \, W(\omega,t)  
\end{equation}
Consider $\gamma \geq 0$ to be chosen later and define the following norm on $C^0 ([0,T]; \R^d)$
\begin{equation}\label{renorm}
\Vert \eta \Vert_\gamma = \max_{t \in [0,T]} e^{-\gamma t} \, \vert \eta (t) \vert, \quad \text{ for any } \eta \in C^0 ([0,T]; \R^d), 
\end{equation}
which turns out to be equivalent to the uniform norm and defines a Banach space structure on $C^0 ([0,T]; \R^d)$. 
We now prove that $\Phi_\omega$ is a contraction. Observe that for any $t \in [0,T]$ 
we have
\begin{align}
\notag &\vert \Phi_\omega (\eta_1)(t) - \Phi_\omega (\eta_2)(t) \vert \leq\\
&\int_0^t \left[\vphantom{\frac{1}{m} \sum_{j=1}^{m}} \left\vert H_1 \ast \nu_s (\eta_1 (s)) - H_1 \ast \nu_s (\eta_2 (s)) \right\vert \right. \\
&\left. + \frac{1}{m} \sum_{j=1}^{m} \left\vert K_1 (\widetilde{Y}^{j} (s) - \eta_1(s)) - K_1 (\widetilde{Y}^{j} (s) - \eta_2(s)) \right\vert \right]\, ds \leq\\
&2 L \int_0^t \vert \eta_1(s)-\eta_2(s)\vert\,ds \leq \frac{2L}{\gamma} \Vert \eta_1-\eta_2\Vert_\gamma e^{\gamma t}.
\end{align}
Notice that in the estimates above $L$ can be taken independent of $s$ thanks to item (ii) in Remark \ref{fuffa}. Multiplying both sides by $e^{-\gamma t}$ and taking the maximum we get
\begin{equation*}
\Vert \Phi_\omega (\eta_1) - \Phi_\omega (\eta_2) \Vert_\gamma \leq \frac{2L}{\gamma} \Vert \eta_1-\eta_2\Vert_\gamma 
\end{equation*}
Taking $\gamma>2L$ we see that $\Phi_\omega$ is a contraction. Hence it admits a unique fixed point, say $\widetilde{X}(\omega)$ that is achieved as usual by reiteration of the map $\Phi_\omega$.
Let us point out that for any $X \in \mathcal{M} (\Omega; C^0 ([0,T]; \R^d))$ that is $\mathcal{F}_t$-adapted, $\omega \mapsto \Phi_\omega(X(\omega))$ clearly belongs to $\mathcal{M} (\Omega; C^0 ([0,T]; \R^d))$ and is $\mathcal{F}_t$-adapted. Thus $\widetilde{X}\in \mathcal{M} (\Omega; C^0 ([0,T]; \R^d))$ and is $\mathcal{F}_t$-adapted, being the limit of $\mathcal{F}_t$-adapted $C^0 ([0,T]; \R^d)$-valued random variables.\\
We now prove the estimate on the $p$-th moment of $\widetilde{X}$.
Observe that $\widetilde{X}$ satisfies the integral equation
\begin{equation*}
	\widetilde{X}(t)=\oX_0+\int_0^t \left[ H_1 \ast \nu_s (\widetilde{X}(s)) + \frac{1}{m} \sum_{j=1}^{m} K_1 (\widetilde{Y}^{j} (s) - \widetilde{X} (s))\right] \, ds+ \sqrt{2 \sigma} \, W(t).
\end{equation*}
It then follows that
\begin{align*}
	|\widetilde{X}(t)|^p&\le C_p\left[ |\oX_0|^p+t^{1-\frac{1}{p}}\int_0^t \left[|H_1 \ast \nu_s (\widetilde{X}(s))|^p + \frac{1}{m} \sum_{j=1}^{m} |K_1 (\widetilde{Y}^{j} (s) - \widetilde{X} (s))|^p\right] \, ds+|W(t)|^p\right]\\
	&\le C_p\left[\vphantom{\frac{1}{m} \sum_{j=1}^{m}}  |\oX_0|^p+t^{1-\frac{1}{p}}\int_0^t \left[\vphantom{\frac{1}{m} \sum_{j=1}^{m}}L^p |\widetilde{X}(s)|^p+|H_1 \ast \nu_s (0)|^p \right. \right.\\ &\left. \left.+ \frac{1}{m} \sum_{j=1}^{m} (1+|\mathfrak{M}_T|^p+ |\widetilde{X} (s)|^p)\right] \, ds+|W(t)|^p\right]
\end{align*}
where $\mathfrak{M}_T=\max_{i=1,\dots,m}\max_{t \in [0,T]}|\widetilde{Y}^i(t)|$. Notice that
\begin{equation*}
	|H_1 \ast \nu_s (0)|\le \int_{\R^d} |H_1(x)|d\nu_s\le L M_1(\nu) + \vert H_1(0) \vert.
\end{equation*}
We end up by obtaining
\begin{align*}
	|\widetilde{X}(s)|^p& \le C\left[\vphantom{\frac{1}{m} \sum_{j=1}^{m}}  1+|\oX_0|^p+\int_0^s|\widetilde{X}(z)|^p \, dz+|W(t)|^p\right] \\
	&\le C\left[\vphantom{\frac{1}{m} \sum_{j=1}^{m}}  1+|\oX_0|^p+\int_0^t \max_{0 \le w \le z}|\widetilde{X}(w)|^p \, dz+\max_{0 \le s \le t}|W(s)|^p\right],
\end{align*}
for any $s \leq t$, where $C$ depends on $p, \ T, \ L, \ M_1(\nu), \ H_1(0), \ \mathfrak{M}_T$. Taking first the maximum on the left-hand side and then the expected values we conclude Step $1$ by Theorem \ref{revuzunz} and Gr\"onwall's inequality.

\textbf{Step 2:} Define $\SSS$ that maps $(\nu,W,Z_0)$ into the pathwise solution of \eqref{micro2}, with components $\SSS^i(\nu,W,Z_0)$, $i=1,..,m+1$, and 
\begin{equation}\label{titti}
\TTT(\cdot,W,Z_0) = {\rm Law} (\SSS^1 (\cdot, W,Z_0)).
\end{equation}
We will show that the map $\TTT$ has a unique fixed point in $\W_p (C^0([0,T];\R^d))$.\\

\textbf{Proof of Step 2:}
Again we endow $C^0([0,T];\R^d)$ with the norm \eqref{renorm} and we observe that $\W_p (C^0([0,T];\R^d), \Vert \cdot \Vert_\gamma)$ remains unchanged with an equivalent metric denoted by $\W_{p,\gamma}$. \\

For fixed $W$ and $Z_0$ and two given elements $\mu_1,\mu_2 \in \W_p$, omitting the dependence on $W$ and $Z_0$ we have

\begin{equation}\label{wasstre}
\W_{p,\gamma}^p (\TTT(\mu_1),\TTT(\mu_2)) \leq \E \left( \Norm{\widetilde{X}_1 - \widetilde{X}_2}{\gamma}^p \right)
\end{equation}

where $\widetilde{X}_i = \SSS^1 (\mu_i, W. Z_0)$, $i=1,2$. Denoting $\widetilde{Y}_i^j=\SSS^{j+1} (\mu_i, W,Z_0)$, $i=1,2$, $j=1,..,m$ we have

\begin{align}\label{icsilon}
&\vert \widetilde{X}_1(t) - \widetilde{X}_2(t) \vert \leq\\
&\int_0^t \left[\vphantom{\frac{1}{m} \sum_{j=1}^{m}} \left\vert H_1 \ast (\mu_1)_s (\widetilde{X}_1 (s)) - H_1 \ast (\mu_2)_s (\widetilde{X}_2 (s)) \right\vert \right. \\
&\left. + \frac{1}{m} \sum_{j=1}^{m} \left\vert K_1 (\widetilde{Y}_1^{j} (s) - \widetilde{X}_1) - K_1 (\widetilde{Y}_2^{j} (s) - \widetilde{X}_2) \right\vert \right]\, ds.\\
&\leq 2L \int_0^t \left[\vert \widetilde{X}_1 (s) - \widetilde{X}_2 (s) \vert  + \max_{j=1,..,m} \vert \widetilde{Y}_1^j (s) - \widetilde{Y}_2^j (s) \vert \right] \, ds\\ 
&+ \int_0^t \left\vert H_1 \ast (\mu_1-\mu_2)_s (\widetilde{X}_1 (s)) \right\vert \, ds\,.
\end{align}

We now prove the estimate for $Y^j$:

\begin{align}
\label{ypsilon} &\max_{j=1,..,m} \vert \widetilde{Y}_1^j (t) - \widetilde{Y}_2^j (t) \vert \leq\\
&\int_0^t \max_{j=1,..,m} \left[\vphantom{\frac{1}{m} \sum_{j=1}^{m}} \left\vert K_2 \ast (\mu_1)_s (\widetilde{Y}_1^j (s)) - K_2 \ast (\mu_2)_s (\widetilde{Y}_2^j (s)) \right\vert \right. \\
&\left. + \frac{1}{m} \sum_{l=1}^{m} \left\vert H_2 (\widetilde{Y}_1^{l} (s) - \widetilde{Y}_1^j (s)) - H_2 (\widetilde{Y}_2^{l} (s) - \widetilde{Y}^j_2 (s)) \right\vert+|u^j(s,(\mu_1)_s)-u^j(s,(\mu_2)_s)| \right]\, ds\\
&\leq 3L \int_0^t \max_{j=1,..,m} \vert \widetilde{Y}_1^j (s) - \widetilde{Y}_2^j (s) \vert  \, ds\\
&+\max_{j=1,..,m} \int_0^t \left\vert K_2 \ast (\mu_1-\mu_2)_s (\widetilde{Y}_1^j (s)) \right\vert \, ds\, + L\int_0^t \W_1((\mu_1)_s,(\mu_2)_s)\, ds \, .
\end{align}

We now point out the following estimate 
\begin{equation}\label{wasszero}
\left(\int_0^t \W_1 ((\mu_1)_s,(\mu_2)_s) \, ds\right)^p \leq  \frac{T^{1-\frac1p}}{p \gamma} \ \W_{p,\gamma}^p (\mu_1,\mu_2)\ e^{p \gamma t}.
\end{equation}
Indeed, consider the optimal coupling $\pi$ of $\mu_1,\mu_2$ for the $\W_{p,\gamma}$ distance and define $\pi_t=({\rm ev}_t)_{\sharp}\pi$. Then $\pi_t$ is a coupling of $(\mu_1)_t$ and $(\mu_2)_t$ and, by definition of $\W_1$, we get

\begin{align*}
	\left(\int_0^t \W_1 ((\mu_1)_s,(\mu_2)_s) \, ds\right)^p&\le T^{1-\frac{1}{p}}\int_0^t \left(\int_{(\R^d)^2}|x-y| \, d\pi_s(x,y)\right)^p \, ds\\
	&\le T^{1-\frac{1}{p}}\int_0^t \int_{(\R^d)^2}|x-y|^p \, d\pi_s(x,y) \, ds \\
	&\le T^{1-\frac{1}{p}}\int_0^t \int_{(C^0([0,T]; \R^d))^2}{\rm ev}_s(|\varphi_1-\varphi_2|^p) \, d\pi(\varphi_1,\varphi_2) \, ds\\
	&\le T^{1-\frac{1}{p}}\int_0^t e^{p \gamma s}\int_{(C^0([0,T]; \R^d))^2}\Norm{\varphi_1-\varphi_2}{\gamma}^{p} \, d\pi(\varphi_1,\varphi_2) \, ds\\
	&\le \frac{T^{1-\frac1p}}{p \gamma} \ \W_{p,\gamma}^p (\mu_1,\mu_2)\ e^{p \gamma t}.
\end{align*}
In particular, \eqref{wasszero} implies that 

\begin{align} 
&\max\left\{\left(\int_0^t \left\vert H_1 \ast (\mu_1-\mu_2)_s (\widetilde{X}_1 (s)) \right\vert \, ds\right)^p, \left(\int_0^t \left\vert K_2 \ast (\mu_1-\mu_2)_s (\widetilde{Y}_1^j (s)) \right\vert \, ds\right)^p \right\}\notag \\& \qquad \qquad \qquad \qquad \qquad \qquad \qquad \qquad \qquad \qquad \leq  \frac{L^p}{p \gamma} T^{1-\frac1p} \ \W_{p,\gamma}^p (\mu_1,\mu_2)\ e^{p \gamma t} \label{wassuno}
\end{align}

Using \eqref{ypsilon}, \eqref{wasszero} and  \eqref{wassuno}, setting 
$$G_1(t) = \max_{j=1,..,m} \vert \widetilde{Y}_1^j (t) - \widetilde{Y}_2^j (t) \vert^p$$ 

we get that

\begin{equation*}
G_1(t) \leq 3^{p-1} L^p T^{1-\frac1p} \left[3^p \int_0^t G_1(s) \, ds + \frac{2}{p \gamma} \ \W_{p,\gamma}^p (\mu_1,\mu_2)\ e^{p \gamma t}\right].
\end{equation*}

We can now apply Gr\"onwall inequality in order to obtain

\begin{equation}\label{gr1}
G_1(t) \leq \frac{C(p,L,T)}{\gamma}  \ \W_{p,\gamma}^p (\mu_1,\mu_2)\ e^{p \gamma t} .
\end{equation}


By \eqref{icsilon}, \eqref{gr1} and the first inequality in \eqref{wassuno} we have, for any $s \leq t$

\begin{align}\label{icsilon2}
&\vert \widetilde{X}_1(s) - \widetilde{X}_2(s) \vert^p \leq\\
\leq 3^{p-1} T^{1-\frac1p} &\left( (2L)^p  \int_0^s e^{p \gamma z} e^{-p \gamma z} \vert \widetilde{X}_1 (z) - \widetilde{X}_2 (z) \vert^p \, dz \right.\\  
&\left. + \W_{p,\gamma}^p (\mu_1,\mu_2) \frac{e^{p \gamma s}}{p \gamma} \left( \frac{C(p,L,T)}{\gamma} + L^p \right) \right)\\
\leq 3^{p-1} T^{1-\frac1p} e^{p \gamma s} &\left( (2L)^p  \int_0^t \max_{0\leq w \leq z} e^{-p \gamma w} \vert \widetilde{X}_1 (w) - \widetilde{X}_2 (w) \vert^p \, dz\right.\\ 
&\left. +  \frac{\W_{p,\gamma}^p (\mu_1,\mu_2)}{p \gamma} \left( \frac{C(p,L,T)}{\gamma} + L^p \right) \right)
\end{align}

Multiplying by $e^{-p \gamma s}$, taking the maximum in $[0,t]$ on the LHS of the above inequality and applying the expectation to both sides, setting 
$$G_2(t) = \E \left( \max_{0\leq s\leq t} e^{-p \gamma s} \ \vert \widetilde{X}_1 (s) - \widetilde{X}_2 (s) \vert^p \right)$$

we get 
  
\begin{align}\label{icsilon3}
G_2(t) \leq 3^{p-1} T^{1-\frac1p} \left( (2L)^p  \int_0^t G_2(s) \, ds  +  \frac{\W_{p,\gamma}^p (\mu_1,\mu_2)}{p \gamma} \left( \frac{C(p,L,T)}{\gamma} + L^p \right) \right).
\end{align}

We apply Gr\"onwall inequality once again and, assuming $\gamma \geq 1$, we obtain

\begin{equation}\label{gr2}
G_2(t) \leq \frac{C(p,L,T)}{\gamma}  \ \W_{p,\gamma}^p (\mu_1,\mu_2).
\end{equation}

Now we plug the previous inequality in \eqref{wasstre} in order to have

\begin{equation}\label{gr3}
\W_{p,\gamma}^p (\TTT(\mu_1),\TTT(\mu_2)) \leq \frac{C(p,L,T)}{\gamma}  \ \W_{p,\gamma}^p (\mu_1,\mu_2).
\end{equation}

Hence choosing $\gamma$ sufficiently large we get that the map $\TTT$ is a contraction.\\

\textbf{Step 3:} Given $\omu$ the fixed point of Step 2 we construct the solution of equation \eqref{micro} as $\SSS(\omu,W,Z_0)$. The estimate $M_p (\omu) < \infty$ follows directly from Step 1. This concludes the proof.
\end{proof}

\begin{rem}
We point out that it is possible to prove the previous result, as well as the following one, under the milder assumption that the Lipschitz constant of $u^i(t,\cdot)$ depends on $t$, \textit{i.e.} $L=L(t)$, with $L \in L^1(0,T)$.  
\end{rem}

In the sequel we will need stability of the solutions to \eqref{micro} with respect to the controls. We will make an additional structural assumption on the $u^i$'s, namely
\begin{equation}\label{sepvar}
u^i(t,\mu)=h^i(t)g^i(\mu)
\end{equation}
where $h^i \in L^\infty([0,T];\R^{d \times \ell})$ for some $\ell \ge 1$ and $g^i:\W_1(\R^d) \to \R^{\ell}$ is Lipschitz. The product in \eqref{sepvar} is the matrix-vector product.

\begin{thm}\label{thm:stability}
Assume that $h^i_j \rightharpoonup h^i$ in $L^1([0,T];\R^{d \times \ell})$ and $g^i_j \to g^i$ in $C(\W_1(\R^d);\R^{\ell})$. Suppose further that $g^i_j$ are $L$-Lipschitz and $M$-bounded while the $h^i_j$ are $M$-bounded. Define $u^i$ and $u^i_j$ as in \eqref{sepvar} and, accordingly, take $\omu_j$, $\omu$, $\oY^i_j$ and $\oY^i$ as in \eqref{micro} for a fixed initial datum $Z_0=(\oX_0,\oY_0^1,\dots,\oY_0^m)$. Then it holds
\begin{equation}
\lim_{j \to \infty}\max_{0 \le t \le T}(\W_1((\omu_j)_t,\omu_t)+\max_{i=1,\dots,m}|\oY^i_j(t)-\oY^i(t)|)=0. 
\end{equation} 
\end{thm}
\begin{proof}
Fix a Brownian motion $W$ and let $(\oX_j,\oY^1_j,\dots,\oY^m_j)=\SSS(\omu_j,W,Z_0)$ and \linebreak $(\oX,\oY^1,\dots,\oY^m)=\SSS(\omu,W,Z_0)$. Then we have
\begin{equation}\label{Fine}
\max_{0 \le t \le T}\W_1((\omu_j)_t,\omu_t)\le \max_{0 \le t \le T}\E[|\oX_j(t)-\oX(t)|]\le \E\left[\max_{0 \le t \le T}|\oX_j(t)-\oX(t)|\right].
\end{equation}
Arguing as in \eqref{ypsilon} and using \eqref{sepvar} we get, for all $j \in \NN$ and for all $t \in [0,T]$
\begin{align}\label{ypsilon2}
\max_{i=1,\dots,m}|\oY^i_j(t)-\oY^i(t)| &\le 3L\int_0^t \max_{i=1,\dots,m}|\oY^i_j(s)-\oY^i(s)|ds\\
&+L(1+2M)\int_0^t \W_1((\omu_j)_s,\omu_s)ds+\max_{t \in [0,T]}R_j(t),
\end{align}
where the remainder term is given by
\begin{equation}
R_j(t)=\max_{i=1,\dots,m}\left| \int_0^T (h^i_j(s)-h^i(s))g^i(\omu_s)\chi_{[0,t]}(s)ds\right|+MT\max_{i=1,\dots,m} \| g_j^i - g^i \|_{\infty}.
\end{equation}
Observe that, by weak convergence of $h^i_j$ and uniform convergence of the $g_j^i$, $R_j(t)$ converges pointwise to zero. By $M$-boundedness of $h^i_j$ we also have that $R_j$ are equi-Lipschitz hence
\begin{equation}\label{remainder}
\lim_{j \to \infty}\max_{t \in [0,T]}R_j(t)=0.
\end{equation}
Furthermore, by Gr\"onwall inequality on \eqref{ypsilon2} we get
\begin{equation}\label{Fine2}
\max_{i=1,\dots,m}|\oY^i_j(t)-\oY^i(t)|\le \left(L(1+2M)\int_0^t \W_1((\omu_j)_s,\omu_s)ds+\max_{t \in [0,T]}R_j(t)\right)e^{3LT}
\end{equation}
Arguing as in \eqref{icsilon} we have, for all $0 \le s \le t \le T$
\begin{align}
|\oX_j(s)-\oX(s)|&\le 2L \int_0^t \left(|\oX_j(z)-\oX(z)|+\max_{i=1,\dots,m}|\oY^i_j(z)-\oY^i(z)|\right)dz\\
&\quad +L \int_0^t \W_1((\omu_j)_z,\omu_z)dz
\end{align}
Inserting \eqref{ypsilon2} in the previous inequality we obtain
\begin{align}
|\oX_j(s)-\oX(s)|&\le 2L \int_0^t |\oX_j(z)-\oX(z)|dz+L \int_0^t \W_1((\omu_j)_z,\omu_z)dz\\
&+L(1+2M)e^{3LT}\int_0^t\int_0^z\W_1((\omu_j)_w,\omu_w)\, dw \, dz\\
&+Te^{3LT}\max_{t \in [0,T]}R_j(t)
\end{align}
Seting $G_j(t)=\E[\max_{0 \le s \le t}|\oX_j(s)-\oX(s)|]$ and observing that for any $0 \le w \le z$ it holds
\begin{equation}
\W_1((\omu_j)_w,\omu_w)\le \E[|\oX_j(w)-\oX(w)|]\le G_j(z),
\end{equation}
we achieve
\begin{equation}
G_j(t)\le L(3+(1+2M)Te^{3LT}) \int_0^t G_j(z)dz+Te^{3LT}\max_{t \in [0,T]}R_j(t).
\end{equation}
Taking into account \eqref{remainder} and \eqref{Fine}, an application of the Gr\"onwall inequality then gives $$\lim_{j \to \infty}\max_{0 \le t \le T}\W_1((\omu_j)_t,\omu_t)=0.$$
Inserting this into \eqref{Fine2} we conclude the proof.
\end{proof}
\begin{rem}
Notice that we actually proved
\begin{equation*}
\lim_{j \to \infty}\E\left[\max_{0 \le t \le T}|\oX_j(t)-\oX(t)|\right]=0,
\end{equation*}
that implies
\begin{equation*}
\lim_{j \to \infty}\W_1(\omu_j,\omu)=0.
\end{equation*}
\end{rem}

\subsection{Equivalence with a PDE/ODE system}\label{equival}

We now derive the PDE-ODE system satisfied by the $(m+1)$-ple $(\omu_t, \oY^1(t),..,\oY^m (t))$ as obtained above by solving \eqref{micro}.

\begin{prop}\label{sysder}
Let $(\oX (t), \oY^1(t),..,\oY^m (t))$ be the unique solution of \eqref{micro} with given initial data $(\oX_0, \oY_0^1,..,\oY_0^m)$. For $\omu_t$ as in \eqref{micro},
the $(m+1)$-ple $(\omu (t), \oY^1(t),..,\oY^m (t))$ solves the PDE-ODE system 
\begin{equation}\label{limit}
\begin{cases}
(\partial_t - \sigma \Delta) \ \omu_t = - \div \left( (H_1 \ast \overline{\mu}_t (\cdot) + \frac{1}{m} \sum_{j=1}^{m} K_1 (\oY^{j} (t) - \cdot))\omu_t \right)\\
\displaystyle \frac{d}{dt} \overline{Y}^i (t) = K_2 \ast \overline{\mu}_t (\oY^i (t)) + \frac{1}{m} \sum_{j=1}^{m} H_2 (\oY^{j} (t) - \oY^i (t)) + u^i(t,\omu_t). & i=1,..,m\\
\omu_0= \mathrm{Law}(\oX_0), \ \oY^i (0)=\oY^i_0 & i=1,..,m
\end{cases}
\end{equation}
\end{prop}

\begin{proof}
It is clear that the initial data are attained and that the second equation is solved. We only have to derive the first equation in \eqref{limit}. Take any $\varphi \in C^\infty_c (\R^d)$ and letting $\Delta$ indicate the Laplacian in spatial coordinates, applying It\^o's formula \cite[Theorem 4.2.1]{oksendal} we see that

\begin{equation*}
d\varphi (\oX(t)) = \langle \nabla \varphi(\oX(t)), d\oX(t) \rangle + \sigma \Delta\varphi(\oX(t)) \, dt.
\end{equation*}

Inserting the first equation in \eqref{micro} we then get

\begin{align*}
d\varphi (\oX(t)) &= \left\langle \nabla \varphi(\oX(t)), H_1 \ast \overline{\mu}_t (\oX(t)) + \frac{1}{m} \sum_{j=1}^{m} K_1 (\oY^{j} (t) - \oX (t)) \, dt
+ \sqrt{2 \sigma} \, dW(t) \right\rangle\\ 
&+ \sigma \Delta\varphi(\oX(t)) \, dt.
\end{align*}

Integrating the above expression in It\^o's sense we have

\begin{align*}
\varphi (\oX(t)) &= \varphi (\oX_0) + \int_0^t \left\langle \nabla \varphi(\oX(s)), H_1 \ast \overline{\mu}_s (\oX(s)) + \frac{1}{m} \sum_{j=1}^{m} K_1 (\oY^{j} (s) - \oX (s))) \right\rangle  \, ds\\
&+ \int_0^t \sigma \Delta\varphi(\oX(s)) \, ds + \sqrt{2 \sigma} \int_0^t \nabla \varphi (\oX(s)) \cdot \, dW(s).
\end{align*}

When taking the expected values on both sides recall that by \cite[Theorem 3.2.1]{oksendal}

\begin{equation*}
\E \left( \int_0^t \nabla \varphi (\oX(s)) \cdot \, dW(s) \right) = 0.
\end{equation*}

Hence, since $\omu_t = \mathrm{Law} \oX(t)$, we get

\begin{align*}
&\int_{\R^d} \varphi (x) \ d\omu_t(x) = \int_{\R^d} \varphi (x) \ d\omu_0(x)\\  
 &+ \int_0^t \int_{\R^d} \left[ \left\langle \nabla \varphi(x), H_1 \ast \overline{\mu}_s (x) + \frac{1}{m} \sum_{j=1}^{m} K_1 (\oY^{j} (s) - x) \right\rangle + \sigma \Delta\varphi(x) \right] \ d\omu_s(x) \, ds
\end{align*}

as required.
\end{proof}

Before going further in the discussion we wish to point out that under some conditions on the initial datum the $(m+1)$-ple $(\omu_tb, \oY^1(t),..,\oY^m (t))$ provided by \eqref{micro} is indeed the unique solution to \eqref{limit}. To this aim we will make use of the following result proved in \cite{dapraetal} that we state for the reader's convenience in a form suited for our setting.

\begin{lem}\label{dapra}
Consider the equation 
\begin{equation}\label{pde}
(\partial_t - \sigma \Delta) \ \eta_t = - \div \left( V (t,\cdot) \, \eta_t \right),
\end{equation}
with initial datum $\eta_0 = \rho_0 \, dx$. Assume that $\eta_0 \in \W_2 (\R^d)$ and satisfies the finite entropy condition 
\begin{equation*}
\int_{\R^d} \rho_0 (x) \log \rho_0 (x) \, dx < \infty.
\end{equation*}
Assume also that there exists $C>0$ such that 
\begin{equation}\label{eq:sublinear}
	\vert V(t,x) \vert \leq C(1+|x|)
\end{equation}
for all $t \in [0,T]$ and $x \in \R^d$. Then there exists a unique solution $\eta \in C^0([0,T];\mathcal{P}(\R^d))$ of \eqref{pde}. 
\end{lem}

\begin{proof} This is a particular case of Theorem 3.3 in \cite{dapraetal}. Observe that if $V(t,\cdot)$ is sublinear, it satisfies all of the other assumptions in the statement of that theorem.
\end{proof}

We are now ready to state the announced uniqueness result.

\begin{thm}\label{sysderun}
Let $(\oX (t), \oY^1(t),..,\oY^m (t))$ be the unique solution of \eqref{micro} with given initial data $(\oX_0, \oY_0^1,..,\oY_0^m)$. Assume additionally that $\omu_0= \mathrm{Law}(\oX_0)$ belongs to $\W_2(\R^d)$ and is of the form $\omu_0 = \rho_0 \, dx$ where $\rho_0$ has finite entropy. Then, for $\omu(t)$ as in \eqref{micro},
the $(m+1)$-ple $(\omu (t), \oY^1(t),..,\oY^m (t))$ is the unique solution of \eqref{limit} in $C^0([0,T]; \W_1(\R^d) \times (\R^d)^m)$.
\end{thm}

\begin{proof} 
By Theorem \ref{sysder} we already know that $(\omu_t, \oY^1(t),..,\oY^m (t))$ solves \eqref{limit}, so we only need to prove its uniqueness.

Thus, we let $(\widehat{\mu}_t, \widehat{Y}^i(t))$ be another solution of \eqref{limit} with $\widehat{\mu} \in C^0([0,T]; \W_1(\R^d))$ and we define

\begin{equation}\label{velocity}
\widehat{v} (t,x) = H_1 \ast \widehat{\mu}_t (x) + \frac1m \sum_{i=1}^m K_1(\widehat{Y}^i(t)-x).
\end{equation}

Notice that $\widehat{\mu}_t$ solves 

\begin{equation}\label{pde2}
(\partial_t - \sigma \Delta) \ \widehat{\mu}_t = - \div \left( \widehat{v} (t,\cdot) \, \widehat{\mu}_t \right).
\end{equation}

We point out the estimate

$$|\widehat{v}(t,x)| \leq L \left( \max_{t\in[0,T]} M_1(\widehat{\mu}_t) + \max_{i=1,..,m} \max_{t\in[0,T]} \vert \widehat{Y}^i (t) \vert\right) + \vert H_1(x) \vert + \vert K_1(x) \vert$$

By the sublinearity of $H_1$ and $K_1$ this implies that $\widehat{v}$ satisfies condition \eqref{eq:sublinear}. Since the initial datum $\omu_0$ belongs to $\W_2(\R^d)$ and has finite entropy, it follows from Theorem \ref{sysderun} that $\widehat{\mu}_t$ is the unique solution of \eqref{pde2}.

Consider now the SDE 

\begin{equation}\label{microaux}
\begin{cases}
\displaystyle d\,\widehat{X}(t) = \widehat{v}(t,\widehat{X}(t)) \, dt + \sqrt{2 \sigma} \, dW(t)\\
\widehat{X}(0)=\oX_0
\end{cases}
\end{equation}

Thanks to the Lipschitz continuity of $H_1$ and $K_1$ one can see that $\widehat{v}(t,\cdot)$ is globally Lipschitz uniformly with respect to $t \in [0,T]$. Combining this with the sublinearity of $\widehat{v}$, we get that \eqref{microaux} admits a unique strong solution by \cite[Theorem 5.2.1]{oksendal}. Applying It\^o's formula as in the proof of proposition \ref{sysder} we have that $\mathrm{Law}(\widehat{X}(t))$ solves \eqref{pde2}. By uniqueness we then get

\begin{equation}\label{fix}
\mathrm{Law}(\widehat{X}(t)) = \widehat{\mu}_t.
\end{equation}

Combining \eqref{fix} with \eqref{microaux}, by the explicit expression of $\widehat{v}$ in \eqref{velocity} we get that $\widehat{\mu}$ is a fixed point of the contractive map $\TTT$ defined in \eqref{titti}, for $Z_0 = (\oX_0, \oY_0^1,..,\oY_0^m)$. Such is also $\omu$ as seen in Theorem \ref{expath}. By uniqueness of the fixed point it holds $\omu_t = \widehat{\mu}_t$ for all $t \in [0,T]$. Inserting this equality into the remaining ODEs of the system \eqref{limit}, we also get $\oY^i (t) = \widehat{Y}^i(t)$ for all $i=1,..,m$ and all $t \in [0,T]$, by standard uniqueness theory for Cauchy problems. This concludes the proof.
\end{proof}

\subsection{Mean field limit of the discrete system}\label{meanfield}

Throughout the paper we will be concerned with the asymptotic behaviour of the following system of SDEs 

\begin{equation}\label{microdiscr}
\begin{cases}
\displaystyle d\, X^n (t) = \frac{1}{N} \sum_{l=1}^{N} H_1 (X^l(t) - X^n(t)) + \frac{1}{m} \sum_{j=1}^{m} K_1 ( Y^{j} (t) -  X^n (t)) \, dt+ \sqrt{2 \sigma} \, dW^n(t)\\
\displaystyle \frac{d}{dt} Y^i (t) = \frac{1}{N} \sum_{l=1}^{N} K_2 (Y^i (t)-X^l(t)) + \frac{1}{m} \sum_{j=1}^{m} H_2 ( Y^{j} (t) - Y^i (t)) + u_N(X^1(t),..,X^n(t)).\\
X^n(0)= X_0^n, \ Y^i (0)= Y^i_0 ; \quad i=1,..,m, \quad n=1,..,N
\end{cases}
\end{equation}

We will assume that 

\begin{itemize}
\item $X^n_0$ are i.i.d. random variables in $L^p(\Omega)$ for some $p>1$.
\item $W^n$ are independent Brownian motions.
\item For all $N\in \NN$ and  $x^1,..,x^N \in \R^d$, $u_N (x_1,..,x_N) = u \left(\frac{1}{N} \sum_{n=1}^N \delta_{x_n} \right)$ where $u:\W_1(\R^d)\to\R^d$ is a Lipschitz function.  
\end{itemize}

The existence of strong solutions for such systems is guaranteed by \cite[Theorem 5.2.1]{oksendal}.

For the proof of the convergence of the solutions of \eqref{microdiscr} to the solutions of \eqref{micro} we need the following technical lemma.

\begin{lem}
Fix $N \in \NN$ and let $\oX^n_0$ be i.i.d. random variables, $W^n$ be independent Brownian motions, $n=1,..,N$, and $(\oY_0^1,\dots,\oY_0^m) \in (\R^d)^m$. Moreover, let $\oX^n$, $\oY^{i,n}$ be the solution of \eqref{micro} with initial data $Z^n_0=(\oX^n_0,\oY_0^1,\dots,\oY_0^m)$ and Brownian motion $W^n$.
Then 
\begin{enumerate}
\item $\oX^n$ are i.i.d.
\item $\oY^{i,n}=\oY^{i,1}:=\oY^{i}$ for any $i=1,..,m$ and $n=1,..,N$.
\end{enumerate}
\end{lem}

\begin{proof}
First of all let us observe that if we prove $(1)$ then property $(2)$ follows by observing that the second line in \eqref{micro} only depends on the law of $\oX^n$, which in turn is independent of $n$.

We now prove property $(1)$. We begin by noticing that the map $\TTT$ actually  depends only on the law of the initial datum. Indeed, consider $W^j$ and $\oX_0^j$ for $j=1,2$ and define, for $k \geq 1$, 

\begin{align*}
\widehat{X}_0^j &\equiv \oX_0^j,\\
\widehat{X}_k^j (t) &= \oX_0^j+\int_0^t H_1 \ast \nu_s (\widehat{X}_{k-1}^j(s)) \, ds\\
&+ \frac{1}{m} \sum_{i=1}^{m} \int_0^t K_1 (\widetilde{Y}^{i} (s) - \widehat{X}_{k-1}^j (s)) \, ds+ \sqrt{2 \sigma} \, W^j (t)\\
\end{align*}
where $\widetilde{Y}^{i}$ solves the second equation in \eqref{micro2}.
Since we are assuming that $\oX_0^1$ and $\oX_0^2$ are identically distributed, one has inductively that $\mathrm{Law}(\widehat{X}_k^{1})= \mathrm{Law}(\widehat{X}_k^{2})$. Since, from the proof of Theorem \ref{expath} we know that $\widehat{X}_k^{j} \to \widetilde{X}^j=\SSS_1(\nu,W_j,Z_0^j)$ for $j=1,2$, we deduce that $\TTT=\mathrm{Law}(\SSS_1)$ only depends on $\nu$ and $\mathrm{Law}(\oX_0^n)$ as claimed.\\
With this, and using again the proof of Theorem \ref{expath}, for all $n$ we have that $\mathrm{Law}(\oX^n)$ is the unique fixed point of $\TTT(\cdot, \mathrm{Law}(\oX_0^n))$. By uniqueness we deduce that $\oX^n$ are identically distributed.

Finally, let $\omu=\mathrm{Law}(\oX^n)$. Then we can write $\oX^n=\SSS_1(\omu,W^n,Z_0^n)$ which clearly implies the required independence. 
\end{proof}

The previous lemma allows us for a standard derivation of the following \textit{propagation of chaos} result.

\begin{thm}
Let $X^n$,$Y^i$ be the solutions of \eqref{microdiscr}. Let $\oX^n$, $\oY^i$ be the solution of \eqref{micro} with initial data $X^n_0$ and Brownian motion $W^n$. Then there exists a constant $C_{d,p}(N)>0$ with $\lim_{N\to\infty} C_{d,p}(N)=0$ such that 
\begin{equation}\label{propchaos}
\E\left( \max_{1 \leq n \leq N} \max_{0\leq t \leq T} \vert X^n(t)-\oX^n (t)\vert + \max_{1\leq i \leq m} \max_{0\leq t \leq T} \vert Y^i(t) - \oY^i(t) \vert \right) \leq C_{d,p} (N).
\end{equation}
\end{thm}

\begin{proof}
Set $\mu_N$ and $\omu_N$ to be the empirical measures  of $X^n$ and $\oX^n$, respectively. For all $0 \leq s \leq t \leq T$ we have by definition of strong solution

\begin{align}
X^n(s) - \oX^n(s) &= \int_0^s \left[(H_1 \ast (\mu_N)_z) (X^n(z)) - (H_1 \ast \omu_z) (\oX^n(z)) \right] \, dz\\
&+ \frac{1}{m} \sum_{j=1}^{m} \int_0^s  \left[K_1 ( Y^{j} (z) -  X^n (z)) - K_1 ( \oY^{j} (z) -  \oX^n (z)) \right] \, dz 
\end{align}

Arguing as in \eqref{icsilon}, thanks to triangle inequality, we get

\begin{align} 
\notag\max_{1 \leq n \leq N} \vert X^n(s) - \oX^n(s)\vert &\leq 2L \int_0^s \left[ \max_{1 \leq n \leq N} \vert X^n(z)- \oX^n(z) \vert + \max_{1 \leq i \leq m} \vert Y^i(z)- \oY^i(z) \vert \right] \, dz\\
&+ L \int_0^s  \left[ \W_1((\mu_N)_z,(\omu_N)_z) + \W_1((\omu_N)_z,\omu_z) \right] \, dz \label{contrx}
\end{align}

The term $\max_{1 \leq i \leq m} \vert Y^i(z)- \oY^i(z) \vert$ can be estimated in a similar way as in \eqref{ypsilon}, obtaining

\begin{align} \label{contry}
\max_{1 \leq i \leq m} \vert Y^i(s) - \oY^i(s)\vert &\leq 3L \int_0^s \left[\max_{1 \leq i \leq m} \vert Y^i(z)- \oY^i(z) \vert \right] \, dz\\
&+ 2L \int_0^s  \left[ \W_1((\mu_N)_z,(\omu_N)_z) + \W_1((\omu_N)_z,\omu_z) \right] \, dz, 
\end{align}
where we also used the triangle inequality for $\W_1$. 

By definition one has almost surely

\begin{equation}\label{banale}
\W_1((\mu_N)_z,(\omu_N)_z) \leq \frac{1}{N} \sum_{n=1}^{N} \vert X^n(z)-\oX^n(z)\vert \leq  \max_{1 \leq n \leq N} \vert X^n(z) - \oX^n(z)\vert 
\end{equation}

On the other hand, by Theorem \ref{rate} and property (i) of Remark \ref{fuffa}, we have

\begin{equation}\label{fuffa2}
\E(\W_1((\omu_N)_z,\omu_z)) \le C^\prime_{d,p}(N) M_p(\omu_z) \le C^\prime_{d,p}(N) M_p(\omu),
\end{equation}
where $M_p(\omu)<\infty$ by Theorem \ref{expath} and $C^\prime_{d,p}(N) \to 0$ as $N \to \infty$ by Theorem \ref{rate}.

Set 

$$G(t):= \E\left( \max_{1 \leq n \leq N} \max_{0\leq s \leq t} \vert X^n(t)-\oX^n (t)\vert + \max_{1\leq i \leq m} \max_{0\leq s \leq t} \vert Y^i(t) - \oY^i(t) \vert \right).$$

Thanks to \eqref{contrx}, \eqref{contry}, \eqref{banale}, \eqref{fuffa2} and Gr\"onwall inequality we have

$$G(t) \leq 3L C^\prime_{d,p}(N)T M_p(\omu) e^{5LT}.$$

Setting $C_{d,p}(N)=3L C^\prime_{d,p}(N)T M_p(\omu) e^{5LT}$ we conclude the proof.
\end{proof}
\begin{rem}\label{rempropchaos}
We also Remark that, combining inequalities \eqref{banale} and \eqref{fuffa2} with \eqref{propchaos} we also conclude that
\begin{equation*}
\mathbb{E}\left(\max_{0 \le t \le T}\W_1((\mu_N)_t,\omu_t)\right)\le C_{d,p}(N).
\end{equation*}
\end{rem}

\section{Mean field sparse optimal control}\label{optcon}

\subsection{The finite dimensional optimal control problem}\label{finitedim}

In what follows we shall fix a measure $\mu_0 \in \W_p(\R^d)$ for some $p>1$ and a vector $\mathbf{Y}_0 \in (\R^d)^m$.\\

Let us consider a continuous Lagrangian function $\LL:[0,T] \times (\R^d)^m \times \W_1(\R^d) \to \R$ and a continuous control cost $\Psi: (\R^{d \times \ell})^m \times (\R^{\ell})^m \to \R$ that is convex in the first variable. 
We fix a compact set $\UU$ of controls in $\R^{d \times \ell}$ and, for given $M,L>0$, we denote with $\mathcal{G}$ the compact subset of $C^0(\W_1(\R^d);\R^\ell)$ made up of $M$-bounded and $L$-Lipschitz functions. Accordingly  we set $E=L^1([0,T];\mathcal{U}) \times \mathcal{G}$ and, for $(\mathbf{h},\mathbf{g}) \in E^m$, we define the sequence of functionals

\begin{equation}\label{contdisc}
\mathcal{F}_N(\mathbf{h},\mathbf{g})=\mathbb{E}\left[\int_0^T \mathcal{L}\left(t,\mathbf{Y}(t),\frac{1}{N}\sum_{n=1}^{N}\delta_{X^n(t)}\right)dt+\int_0^T \Psi\left(\mathbf{h}(t),\mathbf{g}\left(\frac{1}{N}\sum_{n=1}^{N}\delta_{X^n(t)}\right)\right)dt\right],
\end{equation}
where $\mathbf{Y}=(Y^1,\dots,Y^m)$ and $X^1,\dots,X^N$ solve \eqref{microdiscr} with initial data $\mathbf{Y}_0$ and $X_0^n$, for $n=1,\dots,N$, i.i.d. with ${\rm Law}(X_0^1)=\mu_0$ under the controls \linebreak $u^j_N(X^1,\dots,X^N)=h^j(t)g^j\left(\frac{1}{N}\sum_{n=1}^{N}\delta_{X^n(t)}\right)$ for any $j=1,\dots,m$.

In order to prove the well posedness of the minimizing problem, we need the following stability result.

\begin{prop}\label{prop:stabdiscr}
Assume that $h^i_j \rightharpoonup h^i$ in $L^1([0,T];\R^{d \times \ell})$ and $g^i_j, g^i \in C^0(\W_1(\R^d);\R^{\ell})$ with $g^i_j \to g^i$ pointwise on $\W_p (\R^d)$ for some $p>1$. Suppose further that $g^i_j$ are $L$-Lipschitz and $M$-bounded while the $h^i_j$ are $M$-bounded. Define $u^i$ and $u^i_j$ as in \eqref{sepvar} and, accordingly, take $(X^1_j,..,X^N_j,Y_j^1,..,Y_j^m)$, $(X^1,..,X^N,Y^1,..,Y^m)$ as in \eqref{microdiscr}. Then it holds
\begin{equation}
\lim_{j \to \infty}\E\left[\max_{0 \le t \le T}\left(\W_1\left(\frac{1}{N}\sum_{n=1}^{N}\delta_{X^n_j(t)},\frac{1}{N}\sum_{n=1}^{N}\delta_{X^n(t)}\right)+\max_{i=1,\dots,m}\max_{0 \le t \le T}|Y^i_j(t)-Y^i(t)|\right)\right]=0. 
\end{equation} 
\end{prop}

\begin{proof}
Arguing as in Theorem \ref{thm:stability} for any $0 \leq s \leq t$ we get 

\begin{align}\label{stable1}
&\max_{i=1,\dots,m}|Y^i_j(s)-Y^i(s)| \le 3L\int_0^t \max_{i=1,\dots,m}\max_{w \leq z}|Y^i_j(z)-Y^i(z)|dz\\
\notag &+L(1+2M)\int_0^t \max_{w \leq z}\W_1\left(\frac{1}{N}\sum_{n=1}^{N} \delta_{X^n_j (w)},\frac{1}{N}\sum_{n=1}^{N} \delta_{X^n (w)} \right)\,dz+\sum_{i=1}^{m}\max_{s \leq T}\widehat{R}_j^i(s),
\end{align}
where 

\begin{align}\label{stable2}
&\widehat{R}_j^i (s) = \left| \int_0^s \left( h_j^i (z) - h^i(z) \right) g^i\left(\frac{1}{N}\sum_{n=1}^{N} \delta_{X^n (z)} \right) \, dz \right| \\
&+ M \int_0^T \left| g^i_j\left( \frac{1}{N}\sum_{n=1}^{N} \delta_{X^n (z)} \right)- g^i\left( \frac{1}{N}\sum_{n=1}^{N} \delta_{X^n (z)}\right)\right| \, dz.
\end{align} 

Taking the supremum and the expectation in \eqref{stable1} we obtain

\begin{align}\label{stable3}
&\E \left( \max_{s \leq t} \max_{i=1,\dots,m}|Y^i_j(s)-Y^i(s)|\right) \le 3L\int_0^t \E \left(\max_{i=1,\dots,m}\max_{w \leq z}|Y^i_j(z)-Y^i(z)|\right)dz\\
\notag &+L(1+2M)\int_0^t \E\left(\max_{w \leq z}\W_1\left(\frac{1}{N}\sum_{n=1}^{N} \delta_{X^n_j (w)},\frac{1}{N}\sum_{n=1}^{N} \delta_{X^n (w)} \right)\right)\,dz+R_j,
\end{align}
where $R_j= \E \left(\sum_{i=1}^{m}\max_{s \leq T}\widehat{R}_j^i(s)\right)$.

Once we prove that $R_j \to 0$ we conclude the proof as in Theorem \ref{thm:stability}. To do this observe that the first integral in \eqref{stable2} goes to zero almost surely due to the weak convergence of $h_j^i$. Concerning the second integral, observe that due to the pointwise convergence of the $g_j^i$ we know that 
$$\left| g^i_j\left( \frac{1}{N}\sum_{n=1}^{N} \delta_{X^n (z)} \right)- g^i\left( \frac{1}{N}\sum_{n=1}^{N} \delta_{X^n (z)}\right)\right| \to 0$$

for any $z \in [0,T]$ almost surely. Combining this with the fact that $g_j^i$ are $M$-bounded we conclude by dominated convergence that the second integral also goes to zero almost surely. Let us also stress that, since $g_j^i$ and $h_j^i$ are $M$-bounded, $\widehat{R}_j^i (s)$ are equi-Lipschitz and equi-bounded, thus $\max_{s \leq T}\widehat{R}_j^i(s) \to 0$ almost surely and $R_j \to 0$ by dominated convergence. 
\end{proof}

Now we can prove that the functional $\mathcal{F}_N$ admits a minimizer in $E^m$.

\begin{thm}\label{buonaposdiscr}
Let $\mathcal{L}$ be uniformly continuous. There exists $(\mathbf{h}_*,\mathbf{g}_*) \in E^m$ such that 

\begin{equation}
\mathcal{F}_N (\mathbf{h}_*,\mathbf{g}_*) = \min_{(\mathbf{h},\mathbf{g})\in E^m} \mathcal{F}_N (\mathbf{h},\mathbf{g}).
\end{equation} 
\end{thm}

\begin{proof}
Let $(\mathbf{h}_j, \mathbf{g}_j)\in E^m$ be a minimizing sequence. Being $\mathcal{U}$ a compact subset of $\R^{d \times \ell}$, there exists $\mathbf{h}_*$ such that $\mathbf{h}_j \rightharpoonup \mathbf{h}_*$. Let us also observe that, being $X^n_j$ almost surely continuous for any $n=1,..,N$ and $j \in \NN$, it holds $\frac{1}{N} \sum_{n=1}^{N} \delta_{X_j^n (t)} \in \W_p (\R^d)$ almost surely for any $p \geq 1$. Hence we can suppose that $g_j^i$ are defined on $\W_p (\R^d)$ for some $p>1$. Recalling that $\W_p (\R^d)$ is a $\sigma$-compact dense subset of $\W_1 (\R^d)$, by the Ascoli-Arzel\'a Theorem we can suppose there exists a function $\mathbf{g}_* \in (C(\W_1(\R^d);\R^\ell))^m$ that is $L$-Lipschitz, $M$-bounded and such that $g_j^i \to g_*^i$ pointwise in $\W_p(\R^d)$ and uniformly on any compact subset of $\W_p(\R^d)$. Hence, we are left to prove that
\begin{equation}\label{disczero}
\inf_{(\mathbf{h},\mathbf{g}) \in E^m}\mathcal{F}_N(\mathbf{h},\mathbf{g})=\liminf_{j \to \infty}\mathcal{F}_N(\mathbf{h}_j,\mathbf{g}_j)\ge \mathcal{F}_N(\mathbf{h}_*,\mathbf{g}_*).
\end{equation}

Let $(X^1_j,..,X^N_j,Y^1_j,..,Y^m_j)$ and $(X^1_*,..,X^N_*,Y^1_*,..,Y^m_*)$ be the solutions to \eqref{microdiscr} with controls $u_j= h_j g_j$ and $u_* = h_* g_*$ respectively. Since $\mathcal{L}$ is uniformly continuous there exists a concave modulus of continuity $\omega_1$ such that, almost surely

\begin{align*}
&\int_0^T \left\vert \mathcal{L} \left( t , \mathbf{Y}_j(t), \frac{1}{N} \sum_{n=1}^N \delta_{X_j^n(t)} \right) - \mathcal{L} \left( t , \mathbf{Y}_*(t), \frac{1}{N} \sum_{n=1}^N \delta_{X_*^n(t)} \right) \right\vert \, dt\\
&\leq T \omega_1 \left( \max_{0 \leq t \leq T} \max_{i=0,..,m} \left\vert Y_j^i (t) - Y_*^i (t) \right\vert +  \max_{0 \leq t \leq T} \W_1 \left( \frac{1}{N} \sum_{n=1}^N \delta_{X_j^n(t)}, \frac{1}{N} \sum_{n=1}^N \delta_{X_*^n(t)} \right) \right)
\end{align*}

Taking the expectation and using Jensen's inequality we obtain:

\begin{align}\label{discuno}
&\E \left( \int_0^T \left\vert \mathcal{L} \left( t , \mathbf{Y}_j(t), \frac{1}{N} \sum_{n=1}^N \delta_{X_j^n(t)} \right) - \mathcal{L} \left( t , \mathbf{Y}_*(t), \frac{1}{N} \sum_{n=1}^N \delta_{X_*^n(t)} \right) \right\vert \, dt \right)\\
&\leq T \omega_1 \left( \E \left( \max_{0 \leq t \leq T} \max_{i=0,..,m} \left\vert Y_j^i (t) - Y_*^i (t) \right\vert +  \max_{0 \leq t \leq T} \W_1 \left( \frac{1}{N} \sum_{n=1}^N \delta_{X_j^n(t)}, \frac{1}{N} \sum_{n=1}^N \delta_{X_*^n(t)} \right) \right)\right),
\end{align}

where the RHS tends to zero thanks to Proposition \ref{prop:stabdiscr}.\\

Concerning the control cost, let us observe that, being $\mathbf{h}_j$ $M$-bounded and $\mathbf{g}_j$ both $M$-bounded and $L$-Lipschitz, we can consider a concave modulus of continuity $\omega_2$ for $\Psi$ in $(B_M (0))^{2m}$ so that

\begin{align*}
&\int_0^T \left\vert \Psi \left( \mathbf{h}_j (t), \mathbf{g}_j\left(\frac{1}{N} \sum_{n=1}^N \delta_{X_j^n(t)}\right) \right) - \Psi \left( \mathbf{h}_j (t), \mathbf{g}_j \left(\frac{1}{N} \sum_{n=1}^N \delta_{X_*^n(t)}\right) \right) \right\vert \, dt\\
&\leq T \omega_2 \left( L \max_{0 \leq t \leq T} \W_1 \left( \frac{1}{N} \sum_{n=1}^N \delta_{X_j^n(t)}, \frac{1}{N} \sum_{n=1}^N \delta_{X_*^n(t)} \right) \right)
\end{align*}

Taking again the expectation and using Jensen's inequality we obtain:

\begin{align}\label{discdue}
&\E \left( \int_0^T \left\vert \Psi \left( \mathbf{h}_j (t), \mathbf{g}_j\left(\frac{1}{N} \sum_{n=1}^N \delta_{X_j^n(t)}\right) \right) - \Psi \left( \mathbf{h}_j (t), \mathbf{g}_j \left(\frac{1}{N} \sum_{n=1}^N \delta_{X_*^n(t)}\right) \right) \right\vert \, dt \right)\\
&\leq T \omega_2 \left( L \,\E \left(\max_{0 \leq t \leq T} \W_1 \left( \frac{1}{N} \sum_{n=1}^N \delta_{X_j^n(t)}, \frac{1}{N} \sum_{n=1}^N \delta_{X_*^n(t)} \right) \right) \right),
\end{align}

where the RHS tends to zero thanks to Proposition \ref{prop:stabdiscr}.\\

Next, observe that since $(X_*^1,..,X_*^N)$ admits continuous trajectories almost surely, we know that $\frac{1}{N} \sum_{n=1}^N \delta_{X_*^n(t)}$ belongs to $\W_p (\R^d)$ almost surely for any $p>1$. Since the $\mathbf{g}_j$ converge pointwise to $\mathbf{g}_*$ on $\W_p(\R^d)$, we get

\begin{align*}
&\left\vert \Psi \left( \mathbf{h}_j (t), \mathbf{g}_j \left(\frac{1}{N} \sum_{n=1}^N \delta_{X_*^n(t)}\right) \right) - \Psi \left( \mathbf{h}_j (t), \mathbf{g}_* \left(\frac{1}{N} \sum_{n=1}^N \delta_{X_*^n(t)}\right) \right) \right\vert \\
&\leq \omega_2 \left(\left|  \mathbf{g}_j \left(\frac{1}{N} \sum_{n=1}^N \delta_{X_*^n(t)}\right) -  \mathbf{g}_* \left(\frac{1}{N} \sum_{n=1}^N \delta_{X_*^n(t)}\right) \right| \right) \to 0,
\end{align*}

for any $0 \leq t \leq T$ and almost surely. Moreover, 

\begin{equation*}
\left\vert \Psi \left( \mathbf{h}_j (t), \mathbf{g}_j \left(\frac{1}{N} \sum_{n=1}^N \delta_{X_*^n(t)}\right) \right) - \Psi \left( \mathbf{h}_j (t), \mathbf{g}_* \left(\frac{1}{N} \sum_{n=1}^N \delta_{X_*^n(t)}\right) \right) \right\vert \leq 
2 \|\Psi\|_{L^\infty ((B_M(0))^{2m})}.
\end{equation*}

Hence, by dominated covergence, 

\begin{equation}\label{discrtre}
\E \left( \int_{0}^{T} \left\vert \Psi \left( \mathbf{h}_j (t), \mathbf{g}_j \left(\frac{1}{N} \sum_{n=1}^N \delta_{X_*^n(t)}\right) \right) - \Psi \left( \mathbf{h}_j (t), \mathbf{g}_* \left(\frac{1}{N} \sum_{n=1}^N \delta_{X_*^n(t)}\right) \right) \right\vert \, dt \right) \to 0.
\end{equation}

Finally, by lower semicontinuity of convex integrands with respect to the weak-$L^1$ topology we have

\begin{equation*}
\liminf_{j \to \infty} \int_0^T \Psi \left( \mathbf{h}_j (t), \mathbf{g}_* \left(\frac{1}{N} \sum_{n=1}^N \delta_{X_*^n(t)}\right) \right) \, dt \geq \int_0^T \Psi \left( \mathbf{h}_* (t), \mathbf{g}_* \left(\frac{1}{N} \sum_{n=1}^N \delta_{X_*^n(t)}\right) \right) \, dt
\end{equation*}

almost surely. Observing furthermore that 

\begin{equation*}
\left| \int_0^T \Psi \left( \mathbf{h}_j (t), \mathbf{g}_* \left(\frac{1}{N} \sum_{n=1}^N \delta_{X_*^n(t)}\right) \right) \, dt \right| \leq T \|\Psi\|_{L^\infty ((B_M(0))^{2m})},
\end{equation*}
 
we can use Fatou's Lemma to conclude that 

\begin{align}\label{discrquat}
&\liminf_{j \to \infty} \, \E \left(\int_0^T \Psi \left( \mathbf{h}_j (t), \mathbf{g}_* \left(\frac{1}{N} \sum_{n=1}^N \delta_{X_*^n(t)}\right) \right) \, dt \right)\\
&\geq \E \left( \int_0^T \Psi \left( \mathbf{h}_* (t), \mathbf{g}_* \left(\frac{1}{N} \sum_{n=1}^N \delta_{X_*^n(t)}\right) \right) \, dt \right).
\end{align}

Combining \eqref{discuno}, \eqref{discdue}, \eqref{discrtre}, \eqref{discrquat} with \eqref{contdisc} we get \eqref{disczero}.
\end{proof}

\subsection{The $\Gamma$-limit optimal control problem}\label{gammalim}

We now introduce a model optimal control problem for the PDE-ODE system \eqref{limit}. 

For $(\mathbf{h},\mathbf{g}) \in E^m$, where $E$ is the space introduced in the previous section, for $\LL$ and $\Psi$ as in the previous subsection, we consider  the following functional:

\begin{equation}\label{argmin}
\mathcal{F} (\mathbf{h},\mathbf{g}) = \int_0^T \LL (t,\mathbf{\oY}(t),\omu_t) \, dt + \int_0^T \Psi(\mathbf{h}(t),\mathbf{g}(\omu_t))\, dt, 
\end{equation}

where $\mathbf{\oY}=(\oY^1,\dots,\oY^m) \in (\R^d)^m$ and $\oX_t$ are the solutions of \eqref{micro} corresponding to $u^i(t,\omu_t)=h^i(t)g^i(\omu_t)$ and $\omu_t={\rm Law}(\oX_t)$. We also recall that if ${\rm Law}(\oX_0) \in \W_2(\R^d)$ with density $\rho_0$ with finite entropy, one can equivalently take $(\omu_\cdot,\mathbf{\oY})$ as the unique solutions of \eqref{limit}, which is better suited for the applications.

Our first goal is to prove that $\mathcal{F}$ has a minimum on $E^m$.
\begin{thm}
	There exist $(\mathbf{h}^*,\mathbf{g}^*) \in E^m$ such that
	\begin{equation*}
		\F(\mathbf{h}^*,\mathbf{g}^*)=\min_{(\mathbf{h},\mathbf{g}) \in E^m}\F(\mathbf{h},\mathbf{g}).
	\end{equation*}
\end{thm}
\begin{rem}
	We remark that the control problem we consider is more general than the one in \cite{Fornasier2014}. This latter is the particular case of our analysis which one obtains by assuming $L=0$ and $\Psi(\mathbf{h},\mathbf{g})=\widetilde{\Psi}(\mathbf{u})$ for a convex function $\widetilde{\Psi}$. Indeed, in this case, all the possible $\mathbf{g}$'s are constant vectors (hence independent of $\mu$) and we can recast the minimum problem as an offline control problem of the form
	$$\min_{\mathbf{u} \in L^1([0,T];\, \widetilde{\UU}^m)}\mathcal{F} (\mathbf{u}) = \min_{\mathbf{u} \in L^1([0,T];\,\widetilde{\UU}^m)}\left[\int_0^T \LL (t,\mathbf{\oY}(t),\omu_t) \, dt + \int_0^T \widetilde{\Psi}(\mathbf{u}(t))\, dt\right], $$
	where $\widetilde{\UU} \subseteq \R^d$ is the image of the ball of radius $M$ in $\R^\ell$ through $\UU$.\\
	Conversely, notice that if $\mathcal{U}=\{A\}$ for some $A \in \R^{d \times \ell}$, then the problem reduces to $$\min_{\mathbf{u} \in (C(\W_1(\R^d);\R^{\ell}))^m}\mathcal{F} (\mathbf{u}) = \min_{\mathbf{u} \in (C(\W_1(\R^d);\R^{\ell}))^m}\left[\int_0^T \LL (t,\mathbf{\oY}(t),\omu_t) \, dt + \int_0^T \widetilde{\Psi}(\mathbf{u}(\omu_t))\, dt\right], $$ 
	where $\widetilde{\Psi}(\cdot)=\Psi(A,\dots,A,\cdot)$, that is a purely feedback control problem.
\end{rem}
\begin{proof}
	Consider a minimizing sequence $(\mathbf{h}_j,\mathbf{g}_j)$ in the set $E^m$. We endow $E^m$ with the weak topology of $(L^1([0,T]; \UU))^m$ times the strong topology of $(C^0(\W_1(\R^d); \R^\ell))^m$. Since $\mathcal{U}$ is compact, $\mathbf{h}_j$ is weakly compact in $L^1$. We also notice that $\mathbf{g}_j$ is compact in $(C^0(K; \R^\ell))^m$, for all compact subsets of $\W_1(\R^d)$ by the Ascoli-Arzel\'a theorem. Notice that, by Theorem \ref{expath} for any solution of \eqref{micro} it holds $$\sup_{t \in [0,T]}M_p(\omu_t)\le C(p, \oY^1_0,\dots,\oY_0^1,\omu_0,L),$$ for some $p>1$. This one is a compact subset of $\W_1(\R^d)$, hence, without loss of generality, we can assume that $(\mathbf{h}_j,\mathbf{g}_j)$ is compact in the given topology. We denote by $(\mathbf{h}^*,\mathbf{g}^*)$ a limit point. To prove the theorem we are left to show that
	\begin{equation}\label{liminf0}
		\liminf_{j \to \infty} \F(\mathbf{h}_j,\mathbf{g}_j) \ge \F(\mathbf{h}^*,\mathbf{g}^*).
	\end{equation}
	Let $(\mathbf{\oY}^*(t),\omu_t^*)$ be the solution of \eqref{micro} corresponding to the given initial data and the control pair $(\mathbf{h}^*, \mathbf{g}^*)$. By continuity of $\LL$ and Theorem \ref{thm:stability}, we get
	\begin{equation}\label{liminf1}
		\lim_{j \to \infty} \int_{0}^T \LL (t,\mathbf{\oY}^j(t),\omu^j_t) \, dt=\int_{0}^T \LL (t,\mathbf{\oY}^*(t),\omu^*_t) \, dt.
	\end{equation}
	By continuity of $\Psi$ in the product space, the uniform convergence of $\mathbf{g}_j \to \mathbf{g}^*$ and using again Theorem \ref{thm:stability} we have
	\begin{equation}\label{liminf2}
		\lim_{j \to \infty}\int_0^T |\Psi(\mathbf{h}^j(t),\mathbf{g}^j(\omu^j_t))-\Psi(\mathbf{h}^j(t),\mathbf{g}^*(\omu^*_t))|\, dt=0.
	\end{equation}
	By lower semicontinuity of convex integrands with respect to the weak-$L^1$ topology we have
	\begin{equation}
		\liminf_{j \to \infty} \int_0^T \Psi(\mathbf{h}^j(t),\mathbf{g}^*(\omu^*_t))\, dt\ge \int_0^T \Psi(\mathbf{h}^*(t),\mathbf{g}^*(\omu^*_t))\, dt.
	\end{equation}
	Combining this with \eqref{liminf1} and \eqref{liminf2} we obtain \eqref{liminf0}, by the explicit expression of $\F$ given in \eqref{argmin}.
\end{proof}

\begin{rem}\label{ummagumma}
Let us remark that \eqref{liminf0} still holds under the weaker hypothesis that $\mathbf{g}^j \to \mathbf{g}^*$ in $(C^0(K;\R^\ell))^m$, for any bounded subset $K$ of $\W_p (\R^d)$. Observe that any bounded subset $K \subset \W_p (\R^d)$ is precompact in the topology of $\W_1 (\R^d)$.
\end{rem}

We have the following theorem which generalizes to our setting the results in \cite[Theorem 5.3]{Fornasier2014}.

\begin{thm}\label{gammaconv}
Equip $E$ with the product of the weak topology of $L^1$ and the strong topology of $C^0(\W_1(\R^d),\R^\ell)$. For fixed initial data $\mathbf{Y}_0$ and $\mu_0$, define the functionals $\mathcal{F}_N$ and $\mathcal{F}$ as in \eqref{contdisc} and \eqref{argmin}, respectively. Assume additionally that $\mathcal{L}$ is uniformly continuous on $[0,T] \times (\R^d)^m \times \W_1(\R^d)$. Then $\mathcal{F}_N \overset{\Gamma}{\to} \mathcal{F}$.
\end{thm}
\begin{proof}
We begin with the liminf inequality. Consider a sequence $(\mathbf{h}_N,\mathbf{g}_N)$ converging to $(\mathbf{h},\mathbf{g})$ in $E^m$ and let $\mathbf{Y}_N$ and $X_N^1,\dots,X_N^N$ be the corresponding solutions of \eqref{microdiscr}. Let also $\overline{\mathbf{Y}}_N$ and $(\omu_N)_t$ be the solutions of \eqref{micro} under the controls $u^j_N=h^j_Ng^j_N$ for $j=1,\dots,m$ and with the same initial data. By Remark \ref{rempropchaos} it holds
\begin{equation}\label{speranzaMdS}
\E\left[\max_{t \in [0,T]}\W_1\left(\frac{1}{N}\sum_{n=1}^{N}\delta_{X^n_N(t)},(\omu_N)_t\right)+\max_{t \in [0,T]}|\mathbf{Y}_N(t)-\mathbf{\overline{Y}}_N(t)|\right]\le C_{d,p}(N).
\end{equation}
Let $\omega_1$ be an increasing concave modulus of continuity for $\mathcal{L}$. By Jensen's inequality and \eqref{speranzaMdS} we get
\begin{align*}
&\E\left[\max_{t \in [0,T]}\omega_1\left(\W_1\left(\frac{1}{N}\sum_{n=1}^{N}\delta_{X^n_N(t)},(\omu_N)_t\right)+|\mathbf{Y}_N(t)-\mathbf{\overline{Y}}_N(t)|\right)\right]\\
&\leq \E\left[\omega_1\left(\max_{t \in [0,T]}\W_1\left(\frac{1}{N}\sum_{n=1}^{N}\delta_{X^n_N(t)},(\omu_N)_t\right)+\max_{t \in [0,T]}|\mathbf{Y}_N(t)-\mathbf{\overline{Y}}_N(t)|\right)\right]\\
&\qquad \le 
\omega_1\left(\E \left[\max_{t \in [0,T]}\W_1\left(\frac{1}{N}\sum_{n=1}^{N}\delta_{X^n_N(t)},(\omu_N)_t\right)+\max_{t \in [0,T]}|\mathbf{Y}_N(t)-\mathbf{\overline{Y}}_N(t)|\right]\right)\le \omega_1(C_{d,p}(N)).
\end{align*}
Thus, we achieve
\begin{align*}
&\left|\E\left[\int_0^T \mathcal{L}\left(t,\mathbf{Y}_N(t),\frac{1}{N}\sum_{n=1}^{N}\delta_{X^n_N(t)}\right)dt\right]-\int_0^T \mathcal{L}\left(t,\overline{\mathbf{Y}}_N(t),(\omu_N)_t\right)dt\right|\\
&\qquad \le \E\left[\int_0^T \left|\mathcal{L}\left(t,\mathbf{Y}_N(t),\frac{1}{N}\sum_{n=1}^{N}\delta_{X^n_N(t)}\right)- \mathcal{L}\left(t,\overline{\mathbf{Y}}_N(t),(\omu_N)_t\right)\right|dt\right]\\
&\qquad \le  T\E\left[\max_{t \in [0,T]} \omega_1\left(\W_1\left(\frac{1}{N}\sum_{n=1}^{N}\delta_{X^n_N(t)},(\omu_N)_t\right)+|\mathbf{Y}_N(t)-\mathbf{\overline{Y}}_N(t)|\right)\right] \le T \omega_1(C_{d,p}(N)).
\end{align*}
With a similar argument we get
\begin{align*}
\left|\E\left[\int_0^T \Psi\left(\mathbf{h}_N(t),\mathbf{g}_N\left(\frac{1}{N}\sum_{n=1}^{N}\delta_{X^n_N(t)}\right)\right)dt\right]-\int_0^T \Psi(\mathbf{h}_N(t),\mathbf{g}_N(\omu_N)_t))dt\right|& \\
& \! \! \! \! \! \! \! \le T \omega_2(LC_{d,p}(N)),
\end{align*}
where $\omega_2$ is an increasing concave modulus of continuity for $\Psi$ in $\mathcal{U} \times B_M(0)$ and the constants $L$ and $M$ are those fixed in the definition of $\mathcal{G}$. Combining the two previous inequality, we get
\begin{equation}\label{susbtMKDiscr}
\lim_{N \to \infty} \left|\mathcal{F}_N(\mathbf{h}_N,\mathbf{g}_N)-\mathcal{F}(\mathbf{h}_N,\mathbf{g}_N)\right|=0.
\end{equation}
With \eqref{liminf0} we then get the liminf inequality
\begin{equation*}
\liminf_{N \to \infty}\mathcal{F}_N(\mathbf{h}_N,\mathbf{g}_N)\ge \mathcal{F}(\mathbf{h},\mathbf{g}).
\end{equation*}
Finally, observe that \eqref{susbtMKDiscr} applied to the constant sequence $(\mathbf{h},\mathbf{g})$ implies the pointwise convergence
\begin{equation*}
\lim_{N \to \infty}\mathcal{F}_N(\mathbf{h},\mathbf{g})= \mathcal{F}(\mathbf{h},\mathbf{g}),
\end{equation*}
thus, in particular, the existence of a recovery sequence.
\end{proof}

Observe that the previous $\Gamma$-convergence result is in principle not sufficient to deduce the convergence of the minima of the functionals $\mathcal{F}_N$. This is due to the lack of compactness of $\W_1(\R^d)$ which does not allow us to apply the Arzel\'{a}-Ascoli theorem. However, we are able to overcome this  thanks to the fact that the measures $\omu_N$ associated to $(\mathbf{h}_N,\mathbf{g}_N)$ are equibounded in $\W_p (\R^d)$. Furthermore Remark \ref{ummagumma} ensures that a $\Gamma$-liminf inequality still holds under the local uniform convergence: this will be clarified in the following proposition.

\begin{prop}
For fixed initial data $\mathbf{Y}_0$ and $\mu_0$, define the functionals $\mathcal{F}_N$ and $\mathcal{F}$ as in \eqref{contdisc} and \eqref{argmin}
, respectively.  Assume additionally that $\mathcal{L}$ is uniformly continuous on $[0,T] \times (\R^d)^m \times \W_1(\R^d)$. Then

$$\lim_{N \to \infty} \min_{(\mathbf{h},\mathbf{g})\in E^m} \mathcal{F}_N (\mathbf{h},\mathbf{g}) = \min_{(\mathbf{h},\mathbf{g})\in E^m} \mathcal{F} (\mathbf{h},\mathbf{g}).$$
\end{prop}

\begin{proof}
Let $(\mathbf{h}_N,\mathbf{g}_N)$ be a minimizing sequence for the functional $\mathcal{F}_N$. Being $\mathcal{U}$ a compact subset of $\R^{d\times\ell}$  we  can suppose, without loss of generality, that $\mathbf{h}_N \rightharpoonup \mathbf{h}_*$. Moreover, arguing as in Theorem \ref{buonaposdiscr}, we know there exists $\mathbf{g}_* \in C^0(\W_1(\R^d); \R^\ell)$ such that $\mathbf{g}_N \to \mathbf{g}_*$ in $C^0(K,\R^\ell)$, for any bounded $K \subset \W_p(\R^d)$. Also consider any $(\mathbf{h},\mathbf{g})\in E^m$ and recall that we have shown in Theorem \ref{gammaconv} that $\lim_{N\to\infty}  \mathcal{F}_N (\mathbf{h},\mathbf{g}) =  \mathcal{F} (\mathbf{h},\mathbf{g})$. With this, using Remark \ref{ummagumma}, we have 

\begin{equation}
\mathcal{F} (\mathbf{h}_*,\mathbf{g}_*) \leq \liminf_{N \to \infty} \mathcal{F}_N (\mathbf{h}_N,\mathbf{g}_N) \leq \limsup_{N \to \infty} \mathcal{F}_N (\mathbf{h}_N,\mathbf{g}_N) \leq \lim_{N \to \infty} \mathcal{F}_N (\mathbf{h},\mathbf{g}) = \mathcal{F} (\mathbf{h},\mathbf{g}).
\end{equation}

We conclude the prooof observing that $(\mathbf{h},\mathbf{g}) \in E^m$ is arbitrary.

\end{proof}

\bibliographystyle{abbrv}
\bibliography{biblio} 
\end{document}